\documentclass[a4paper, 11pt]{amsart}
\usepackage[pdftex,hidelinks]{hyperref}

\usepackage[english]{babel}

\usepackage[utf8]{inputenc}    
\usepackage{amssymb}
\usepackage{enumerate}
\usepackage{graphicx} 

\usepackage{bbm}
\usepackage{amsmath}
\usepackage{amsthm}

\usepackage[right= 1.85cm, left=1.85cm, a4paper,top=2.6cm, bottom= 2.85cm]{geometry}

\newcommand{\ind}{\mathbbm{1}}

\numberwithin{equation}{section}

\theoremstyle{plain}
\newtheorem{satz}{Theorem}[section]

\newtheorem{exlem}[satz]{Examplemma}
\newtheorem{prop}[satz]{Proposition}

\theoremstyle{remark}
\newtheorem{rmk}[satz]{Remark}
\newtheorem{bsp}[satz]{Example}
\newtheorem{dfn}[satz]{Definition}

\author{Jonas Jalowy}
\address{\footnotesize Jonas Jalowy: {\itshape Paderborn University, Warburger Str.~100, 33098 Paderborn, Germany}}
\email{jjalowy@math.uni-paderborn.de}

\author{Hanna Stange}

\address{\footnotesize Hanna Stange: {\itshape University of M\"{u}nster, Orl\'{e}ans-Ring 10, 48149 M\"{u}nster, Germany}}
\email{hanna.stange@uni-muenster.de}

\keywords{Hyperuniformity, Point process, Covariance structure, Gaussian Field}
\subjclass[2020]{Primary: 60G55 ;  Secondary: 60F05, 60G50}

\title{Box-Covariances of Hyperuniform Point Processes}

\begin{document}
	\begin{abstract}
In this work, we present a complete characterization of the covariance structure of number statistics in boxes for hyperuniform point processes. 
		Under a standard integrability assumption,
		the covariance depends solely on the overlap of the faces of the box. Beyond this assumption, a novel interpolating covariance structure emerges. 
		
		This enables us to identify a limiting Gaussian ``coarse-grained'' process, counting the number of points in large boxes as a function of the box position. Depending on the integrability assumption, this process may be continuous or discontinuous, e.g.~in $d=1$ it is given by an increment process of a fractional Brownian motion.
\end{abstract}


\keywords{Point process, Hyperuniformity,  Covariance structure, Gaussian field}

\maketitle

\section{Introduction}
	A point process $\eta$ on $\mathbb{R}^d$ is called \emph{hyperuniform} if the density fluctuations in boxes $\Lambda_n = [0,n]^d$, measured by $\mathrm{Var}(\eta(\Lambda_n))$, increase at a slower rate than the volume of the box $\lambda_d(\Lambda_n)$, i.e.\
	\begin{align}\label{eq:HU}
		\lim_{n \to \infty} \frac{\mathrm{Var}(\eta(\Lambda_n))}{\lambda_d(\Lambda_n)} = 0.
	\end{align} 
	While the Poisson point process is not hyperuniform, as its variance satisfies $\mathrm{Var}(\eta(\Lambda_n)) = \lambda_d(\Lambda_n)$, well-known examples of hyperuniform point processes include perturbed lattices, the $\mathsf{sine}_2$ process and the Ginibre ensemble (see \cite{Gacs,dereudre2024nonhyperuniformityperturbedlattices,AGZ,ghosh2016fluctuationslargedeviationsrigidity}). For an overview of the topic, we refer to the surveys \cite{Torquato_2003,Coste,LRSurvey}. 
	Given that hyperuniform point processes are characterized by variance asymptotics of growing boxes, this work answers the question: \begin{center}
		\emph{What is the asymptotic (relative) covariance between the number statistics of two boxes?}
	\end{center} 
	More precisely, we consider a hyperuniform point process $\eta$ and study the limiting correlation coefficient
	\begin{align}\label{eq:cov}
		\mathrm{cov}(z):= \lim_{n\to\infty}\frac{\mathrm{Cov}\big( \eta(\Lambda_n),\eta(\Lambda_n(nz))\big)}{\mathrm{Var}\big(\eta(\Lambda_n)\big)}, \quad z\in\mathbb{R}^d,
	\end{align}
	where $\Lambda_n(nz):=nz+\Lambda_n$ is the box of volume $n^d$ starting at $nz\in\mathbb{R}^d$. 
	
	\subsection{Results} 
	Recall that a stationary point process with \emph{truncated pair correlation measure} $\beta$ (see  \eqref{eq:truncatedpaircorrelationmeasure} below for the definition) of finite total variation, i.e. $|\beta|(\mathbb{R}^d)<\infty$, is hyperuniform if and only if  $\beta(\mathbb{R}^d)=-1$, see Proposition \ref{prop:hyperuniformity}. In the sequel, we call $\beta$ \emph{box-symmetric}, if it is invariant under permutation and reflection of its coordinates.
	Assuming not only that $|\beta|(\mathbb{R}^d)<\infty$ but in addition the stronger integrability assumption $ \int|y||\beta|(\mathrm{d} y)<\infty$, consistent with \cite{SWY,YogeshKrish}, the following covariance structure for hyperuniform point processes emerges.
	
\begin{satz}\label{thm:Cov_integrable}
	Let $\eta$ be a stationary hyperuniform point process with box-symmetric truncated pair correlation measure $\beta$ satisfying $\int|y||\beta|(\mathrm{d} y)<\infty$. Then, 
	\begin{align}\label{eq:thm_cov}
		\mathrm{cov}(z)=\begin{cases}-\frac 1 {2d}\lambda_{d-1}(\partial \Lambda_1\cap \partial \Lambda_1(z)),\quad &\text{if } \mathring\Lambda_1\cap \mathring\Lambda_1(z)=\emptyset,\\
			\frac 1 {2d}\lambda_{d-1}(\partial \Lambda_1\cap \partial \Lambda_1(z)),\quad &\text{if }\mathring\Lambda_1\cap \mathring\Lambda_1(z)\neq\emptyset,
		\end{cases}
	\end{align}
	where $\partial \Lambda_1(z)$ is the boundary of $\Lambda_1(z)$ and $\mathring\Lambda_1(z)$ is the interior of $\Lambda_1(z)$. In particular, $\mathrm{cov}(z)$ is zero if the boxes are disjoint or if $z\in\mathring \Lambda_1$, it is 1 if $z=0$, and it is $- 1/ {(2d)}$ if $z\in\mathbb{Z}^d,|z|=1$.
\end{satz}
While for hyperuniform point processes satisfying the assumptions of Theorem~\ref{thm:Cov_integrable} the covariance $\mathrm{cov}(z)$ only depends on the overlap of the boundary $\partial \Lambda_1 \cap \partial \Lambda_1(z)$, the covariance structure changes significantly when the integrability assumption is dropped. To this end, we study point processes on $\mathbb{R}$ for which $H(y):=-\int_0^y \beta([x,\infty))\mathrm{d}x$ 
is regularly varying with parameter $a\in[0,1]$, i.e.\
\begin{align*}
	\lim_{n\to\infty} \frac{H(xn)}{H(n)}=x^a,\quad \text{ for all } x\in\mathbb{R}.
\end{align*}
\begin{satz}\label{thm:Cov_non_int}
	Let $\eta$ be a stationary hyperuniform point process on $\mathbb{R}$ with 
	truncated pair correlation measure $\beta$ such that $|\beta|(\mathbb{R})<\infty$. If

	\quad \emph{(i)} the function $H(y)=-\int_0^y \beta([x,\infty))\mathrm{d}x= \frac{1}{2}\mathrm{Var}(\eta(\Lambda_y))$ is regularly varying with parameter $a\in[0,1],$ then,
	
	\quad \emph{(ii)} for all $z\in\mathbb{R}$ we have  
	\begin{align}\label{ieq:Cov_non_int}
		\mathrm{cov}(z)=
		\frac{|z-1|^{a}}{2}+ \frac{|z+1|^{a}}{2}-|z|^{a},
	\end{align}
    for $a\in (0,1]$, and $\mathrm{cov}(z)=\delta_0(z)-\frac{1}{2} \delta_{-1}(z)-\frac{1}{2}\delta_{+1}(z)$ for $a=0$ (as in \eqref{eq:thm_cov}).
	
	Conversely, if we additionally assume that $\beta([x,\infty))$ is non-positive (or non-negative) for $x$ sufficiently large, then \emph{(i)} and \emph{(ii)} are equivalent. Moreover, if
	
	\quad \emph{(iii)} the limit $\mathrm{cov}(z)<\infty$ 
	exists for $z=1$ and $z=2$,
	
	then both \emph{(i)} and \emph{(ii)} follow.
	In either of these cases, neighboring boxes 
	have $\mathrm{cov}(1)=2^{a-1}-1$.
\end{satz}
To the best of our knowledge, the non-integrable case covered by Theorem~\ref{thm:Cov_non_int} did not appear in the literature, which we will discuss in Section \ref{sec:literature} below. The class of point processes that satisfies the assumption of Theorem \ref{thm:Cov_non_int} interpolates between the integrable setting of Theorem~\ref{thm:Cov_integrable} and the Poisson case. Indeed, examining \eqref{ieq:Cov_non_int}, which is represented in Figure~\ref{fig:graph}, we observe that in the first extremal case $a=0$, the covariance is non-zero only for $z \in \{-1,0,+1\}$, as in Theorem~\ref{thm:Cov_integrable}. In the other extremal case $a=1$, the covariance is given by $\mathrm{cov}(z) = \lambda_1(\Lambda_1 \cap \Lambda_1(z)),$ which is the covariance of a Poisson point process.
Examples satisfying the assumptions of Theorem \ref{thm:Cov_non_int} include the large class of hyperuniform determinantal point processes in $d=1$, such as the famous $\mathsf{sine}_2$ process arising as the limit of eigenvalues of GUE matrices, corresponding to $a=0$, and perturbed lattices for arbitrary $a\in[0,1]$. Examples will be discussed in Section \ref{sec:examples} below. 
\begin{figure}[b]
	\centering
	\includegraphics[width=0.39\linewidth]{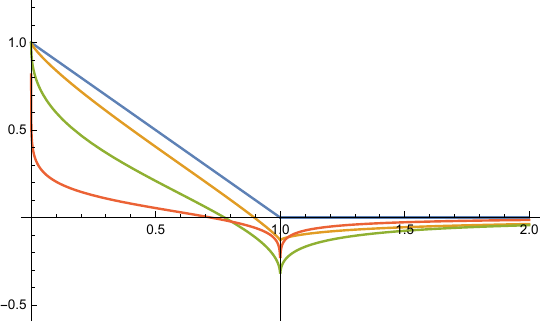}
	\caption{The function $\mathrm{cov}$ from \eqref{ieq:Cov_non_int} describing the covariances between points in intervals shifted by $z$, for five different values of the varying parameter $a\in [0,1]$. The nearly integrable case (black line, $a=0$) is discontinuous in $z=1$. The cases $a=1/10$ (red), $a=4/10$ (green) and $a=8/10$ (yellow) interpolate to the Poisson case $a=1$ (blue).}
	\label{fig:graph}
\end{figure}

\begin{rmk}
	The heuristic reason for negative correlation between neighboring boxes is hyperuniformity: If $\eta(\Lambda_n)$ is large, then we expect the adjacent number $\eta(\Lambda_n(nz))$ to be small in order to maintain hyperuniformity.
	These extra points of $\eta$ inside $\Lambda_n$ will typically lie close to the boundary $\partial\Lambda_n$ (contrary to the Poisson case), which explains the boundary effects. 
	Moreover, in the non-integrable setting of Theorem~\ref{thm:Cov_non_int} long-range interactions induce correlations between distant regions at $|z|>1$. 
\end{rmk}

In higher dimensions $d>1$, the computations and presentation of the results become significantly more involved. Hence, we restrict our discussion in Section \ref{sec:highdim} to the case $ d = 2 $.  The resulting covariance structure is stated in Theorem~\ref{thm:Cov_non_int_2} below and the employed methods can be extended to higher dimensions.

The previous results show that the microscopic correlations (described by $\beta$) become irrelevant when ``zooming out'', and that a universal covariance structure emerges. The number statistics $\eta(\Lambda_n(z))$ forget the microscopic position of points of $\eta$ and what remains in the large-$n$ limit is a ``coarse-grained'' process parametrized by the shift $z\in\mathbb{R}^d$. This ``coarse-grained'' process is asymptotically Gaussian, under the so-called \emph{Brillinger-mixing} condition (see Definition \ref{dfn:Brillinger} below).

\begin{satz}\label{thm:CLT}
	Let $\eta$ be a hyperuniform point process satisfying the assumptions of Theorem \ref{thm:Cov_integrable} for $d\ge2$, Theorem \ref{thm:Cov_non_int} for $d=1$ and $a>0$, or that of Theorem \ref{thm:Cov_non_int_2} below. 
	If $\eta$ is Brillinger-mixing, then we have weak convergence in finite-dimensional distributions
	\begin{align}\label{eq:numberstat}
		\Bigg(\frac{\eta(\Lambda_n(z))-n^d}{\sqrt{\mathrm{Var}\big(\eta(\Lambda_n)\big)}}\Bigg)_{z\in\mathbb{R}^d}{\underset{n\to\infty}\longrightarrow}\big(G(z)\big)_{z\in\mathbb{R}^d} 
	\end{align}
	towards a centered stationary Gaussian field $G$ with covariance function $\mathrm{cov}(z)$. In the case of $d=1$ with $a>0$, we obtain the increment process $G(z)=B_{a/2}(z+1)-B_{a/2}(z)$ of the fractional Brownian motion $B_{a/2}$ with Hurst index $a/2$.
\end{satz}
Note that existence of the Gaussian field $G$ follows from the existence of the fractional Brownian noise, due to Kolmogorov's extension theorem and non-negative-definiteness of $\mathrm{cov}$.
For $ d = 1 $ and $a>0$, the covariance function $\mathrm{cov}$ is continuous at $z=0$ and the Kolmogorov–Chentsov theorem implies that the process $ G $ admits $\alpha$-Hölder continuous paths for all $\alpha < {a}/{2}$, see Figure~\ref{fig:G}. By \cite{NecSuffCont}, the converse implication also holds, i.e.~$G$ has discontinuous paths if ${\mathrm{cov}}$ is discontinuous as in the integrable setting of Theorem \ref{thm:Cov_integrable} or the case $d=1,a=0$ (where $G$ exists by the same reasoning, but Theorem \ref{thm:CLT} does not apply).

\begin{figure}[h]
	\centering
	\includegraphics[width=0.55\linewidth]{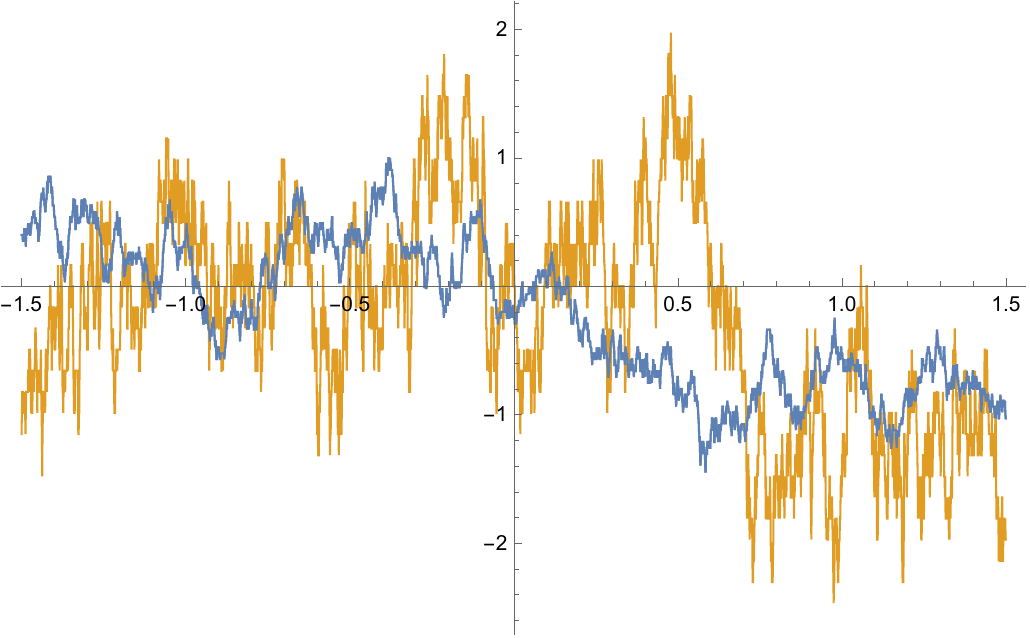}
	\caption{A plot of the rescaled number statistics \eqref{eq:numberstat} for $n=1000$ and $d=1$ (which approximates the ``coarse-grained'' process $G(z)$) for the perturbed $\mathsf{sine}_2$ process with $a=1/4$ (in orange) and $a=3/4$ (in blue), see Example \ref{exlem} (d). Note the higher path regularity in the heavier-tailed case $a=3/4$ and the clearly visible negative correlation at neighboring boxes $\Lambda_n(nz)$, e.g.~at $z\approx-2/5,+3/5$ in blue or $z\approx-1/2,+1/2$ in orange.}
	\label{fig:G}
\end{figure}

\subsection{Connection to the literature}\label{sec:literature}

The study of hyperuniformity goes back to the 1970s (see e.g.\ \cite{Gacs,Martin1980TheCF}), although the term ``hyperuniformity'' first appeared in the 2000s, for instance, in \cite{Torquato_2003}.  In \cite{Torquato_2003} Torquato classified hyperuniform point processes depending on the speed at which \eqref{eq:HU} converges to zero into Class~I at rate $n^{-1}$, Class~II at rate $n^{-1}\log n$, and Class~III at rate $n^{-1+a}$ for some $0<a<1$. From this perspective, the setting of Theorem \ref{thm:Cov_integrable} deals with Class~I hyperuniform systems and Theorem \ref{thm:Cov_non_int} covers the remaining cases (note that this classification is not exhaustive, see \cite[Theorem 3]{dereudre2024nonhyperuniformityperturbedlattices}). By now, hyperuniformity has found many applications for example in chemistry and materials science \cite{Material, PhysRevLett.119.136002} and has become a rapidly growing and active research area, e.g.~the above variance asymptotics have recently been critical in the understanding of hyperuniformity using the language of (optimal) transport, see \cite{lachieze2024hyperuniformity,huesmann2024linkhyperuniformitycoulombenergy,butez2024wasserstein,Klatt}.

In the study of hyperuniformity, the shape of the so-called observation window plays an important role. Although we work under assumptions ensuring that the definition of hyperuniformity is independent of this choice (see Proposition \ref{prop:hyperuniformity}), the limiting covariance structure does depend on the particular window. The most natural observation windows are balls and cubes, see e.g. \cite{KimTorquato}. In this work, we focus on such boxes, since these have non-trivial surface overlap (contrary to balls) and because their product structure allows for direct computations.
Our results extend to rectangular boxes without substantial modifications (see Remarks \ref{rem:rectangles_integrable}, \ref{rem:rectangles_non-integrable}, and \ref{rem:rectangles_CLT}), but we focus on boxes with equal side length since such cubes reveal the same conceptual insight and let us keep a transparent presentation.

The statement of Theorem \ref{thm:Cov_integrable} in the integrable case first appeared as special cases, e.g.~for eigenvalues of random matrices \cite{DiaconisEvans,Wieand}, Coulomb systems \cite{Lebowitz,AGL}, and zeros of the GAF \cite{SodinGAF}. For hyperuniform point processes on the complex plane (making complex analytic tools available and corresponding to $d=2$) and under the same integrability assumption, Theorem \ref{thm:Cov_integrable} follows from Sodin, Wennman, Yakir \cite{SWY}, in which arbitrary domains with rectifiable regular boundary are considered. 
At the time of writing this work, this has been generalized to arbitrary dimensions $d\in\mathbb{N}$ by Krishnapur and Yogeshwaran \cite{YogeshKrish}, also using the same integrability assumption as in Theorem \ref{thm:Cov_integrable}. The authors work under the slightly stronger assumption of $\mathcal C^1$ smooth domains excluding boxes $\Lambda_n$, but conjecture that it should hold for more general domains, say Lipschitz domains. Both proofs of \cite{SWY,YogeshKrish} rely on the smoothness for approximative arguments and on the integrability assumption providing an absolutely continuous spectral measure $\hat \beta$, see e.g.~\cite[Remark 5.3]{SWY1}. From this perspective, our contribution of Theorem \ref{thm:Cov_integrable} is an elementary proof free of approximations, charts, and Fourier analytic methods (for which cubes may cause problems).

The non-integrable setting of Theorem \ref{thm:Cov_non_int} corresponds to point processes that are hyperuniform but not number rigid (see \cite[Corollary 4]{LRigidity}). However, the resulting covariance structure appears to be novel in the existing literature.

Our work is partially motivated by Adhikari, Ghosh, Lebowitz \cite[Theorem 2]{AGL}, who studied stationary random fields on $\mathbb{Z}^d$ under a moment condition corresponding to $\int|y|^d|\beta|(\mathrm{d} y)<\infty$. Their result translates to $\mathrm{cov}(z)=(-2)^{-k}$ for $z\in \{0,1\}^d, |z|=k\in\{0,\dots, d\}$ and $\mathrm{cov}(z)=0$ for all other lattice shifts $ z\in \mathbb{Z}^d$.
In particular, neighboring boxes have negative covariances as in Theorem \ref{thm:Cov_integrable}, but curiously, covariances become positive if the dimension of $\Lambda_1\cap \Lambda_1(z)$ is even. However, their setting is somewhat complementary to that of point processes. 
In particular, Theorem \ref{thm:Cov_integrable} explicitly verifies that the phenomenon of positive covariance between boxes with touching edges is an artifact of the discrete nature of the underlying process.

The works \cite{YogeshKrish, Coste, Mastrilli, Ivanoff, AGL, Lebowitz, SodinGAF} also establish asymptotic normality of $\eta(\Lambda_n)$ or linear statistics $\int \varphi(x/n) \eta(\mathrm{d}x)$ for sufficiently smooth test functions $\varphi$. A powerful and flexible approach for proving such CLTs is the cumulant method, which we will also adopt for the proof of Theorem \ref{thm:CLT}. Under a weakened integrability condition on $\beta$ and for smooth statistics even non-Gaussian limits may emerge, see \cite{mastrilli2026asymptotic}. The assumption $a>0$ in Theorem \ref{thm:CLT} ensures sufficient growth of $\mathrm{Var}(\eta(\Lambda_n))$ and may be weakened using more elaborate CLTs that continue to hold even with vanishing variance, such as by  Krishnapur, Yogeshwaran \cite{YogeshKrish} and Mastrilli \cite{mastrilli2026asymptotic}.

\subsection{Plan of the Paper}
In Section \ref{sec:prelim}, we introduce our main tools, such as the truncated pair correlation measure $\beta$, and collect necessary well-known facts to be used throughout the paper. Section \ref{sec:int} is devoted to the integrable case, where we prove Theorem \ref{thm:Cov_integrable}. The non-integrable case is addressed in Section \ref{sec:nonint}, starting with the case $d=1$, followed by its extension to $d=2$ and illustrated by examples. Finally, in Section \ref{sec:cumulant_CLT}, we turn to the central limit theorem, for which factorial cumulant measures and the Brillinger-mixing condition will be introduced.
\section{Preliminaries}\label{sec:prelim}   
In this section, we introduce the most frequently used notations and recall a well-known covariance formula as well as a criterion for hyperuniformity. For a detailed introduction to the theory of point processes, we refer e.g.\ to \cite{daley2003introduction} and \cite{last_penrose_2017}.\\
Let $\mathsf{N}(\mathbb{R}^d)$ be the space of simple and locally finite counting measures on $\mathbb{R}^d.$  We endow it with the  $\sigma$-algebra $\mathcal{N}(\mathbb{R}^d),$  the smallest $\sigma$-algebra with respect to which the number statistics  $\mu\mapsto\mu(W)$ are measurable for all $W\in\mathcal{B}(\mathbb{R}^d).$ A  \textit{point process} $\eta$ on $\mathbb{R}^d$ is an $\mathsf{N}(\mathbb{R}^d)$-valued random variable.\\
For  $x\in\mathbb{R}^d$, let  $W+x:=\{x+y:y\in W\}$  and  define the shift transformation   $\theta_x:\mathsf{N}(\mathbb{R}^d)\to \mathsf{N}(\mathbb{R}^d)$ via $\theta_x\mu(W):=\mu(W+x)$ for all $W\in\mathcal{B}(\mathbb{R}^d).$ A point process $\eta$ on $\mathbb{R}^d$ is called {stationary} if the distribution is invariant under shifts, i.e.\ $\theta_x\eta\overset{d}{=}\eta$ for all $x\in\mathbb{R}^d.$  \\
The \textit{intensity measure} of a point process $\eta$ on $\mathbb{R}^d$ is defined by 
$\alpha(W) := \mathbb{E}[\eta(W)],$ $ W \in \mathcal{B}(\mathbb{R}^d).$
For  stationary point processes with locally finite intensity measure, it holds that $\alpha=\vartheta\lambda_d$ for some non-negative $\vartheta\in\mathbb{R}_+.$ In this paper, we always assume unit intensity $\vartheta=1$.

Let  $\eta=\sum_{n=1}^\infty\delta_{X_n}$ a.s. The \textit{$k$-th factorial moment  measure} of $\eta$  is defined by \begin{align*}
	\alpha^k(W):=\mathbb{E}\left[\sum_{n_1,\ldots,n_k\in\mathbb{N}}\!{\hspace{-0.3cm}\vphantom{\sum}}^{\not =}\ind_{(X_{n_1},\ldots, X_{n_k})\in W}\right],\quad W\in\mathcal{B}((\mathbb{R}^d)^k).\end{align*} 
The superscript $\not =$  indicates that the summation is only over distinct $n_1,\ldots,n_k.$ For stationary point processes, we define the \textit{reduced $k$-th factorial moment measure} $\alpha^k_!$ via 
\begin{align} \label{eq:shiftpaircorrelation}
	\int f(x_1,\ldots,x_k)\alpha^k(\mathrm{d}x_1,\ldots,\mathrm{d}x_k)=\int\int f(x_1+y,\ldots,x_{k-1}+y,y)\alpha^k_!(\mathrm{d}x_1,\ldots,\mathrm{d}x_{k-1})\mathrm{d}y,
\end{align}
for all non-negative measurable functions $f:(\mathbb{R}^{d})^k\to\mathbb{R}_+$. Using the reduced second factorial moment measure, we define the \textit{truncated pair correlation measure} 
\begin{align}\label{eq:truncatedpaircorrelationmeasure}
	\beta(\mathrm{d} x):=(\alpha_!^2-\alpha)(\mathrm{d} x),
\end{align}
which is a signed measure that describes the correlation of particles at distance $x$. \\
The truncated pair correlation measure can be used to express the covariance between the number of points in boxes of the form $\Lambda_n(z):=[z_1,z_1+n]\times\dots\times[z_d,z_d+n],$ $(z_1,\dots,z_d)=z\in\mathbb{R}^d.$ Recall that we briefly write $\Lambda_n$ instead of $\Lambda_n(0).$ 
\begin{prop}\label{prop:var}
	Let $\eta$ be a stationary point process with truncated pair correlation measure $\beta.$ It holds that 
	\begin{align}\label{eq:Cov}
		\mathrm{Cov} (\eta(\Lambda_n),\eta(\Lambda_n(nz)))= \lambda_d(\Lambda_n\cap\Lambda_n(nz)) + \int_{\Lambda_n}\beta(\Lambda_n(nz-x))\lambda_d(\mathrm{d}x).
	\end{align}
\end{prop}
\begin{proof}
	Let $f,g:\mathbb{R}^d\to\mathbb{R}_+$ be positive-valued measurable functions. Using Campbell's formula
	\begin{align*}\mathbb{E}\left[\int_{\mathbb{R}^d}f(x)\eta(\mathrm{d}x)\right]=\int_{\mathbb{R}^d}f(x)\alpha({\mathrm{d}}x),
	\end{align*}
	the Campbell-type formula of the second order 
	\begin{align*} 
		\mathbb{E}\left[\int_{\mathbb{R}^d}f(x)\eta({\mathrm{d}}x)\int_{\mathbb{R}^d}g(y)\eta({\mathrm{d}}y)\right]=\int_{\mathbb{R}^d}f(x)g(x)\alpha({\mathrm{d}}x)+\int_{\mathbb{R}^d\times\mathbb{R}^d}f(x)g(y)\alpha^2({\mathrm{d}}x,\mathrm{d}y),
	\end{align*} and equation \eqref{eq:shiftpaircorrelation}, we obtain that
	
	\begin{align} \label{eq:400}
		\mathrm{Cov}\left[\int f(x)\eta(\mathrm{d}x), \int g(x)\eta(\mathrm{d}x)\right]=\int_{\mathbb{R}^d}f(x)g(x)\alpha({\mathrm{d}}x)+\int_{\mathbb{R}^d}\int_{\mathbb{R}^d}f(x)g(y+x)\beta(\mathrm{d}y)\alpha(\mathrm{d}x).
	\end{align}
	Inserting  $f=\ind_{\Lambda_n}$ and 
    $g=\ind_{\Lambda_n(nz)}$ in this formula proves  the claim.
\end{proof}

We finish this section by recalling a well-known criterion for hyperuniformity (see e.g.~\cite{Martin1980TheCF, Coste}), using the {truncated pair correlation measure}. 
\begin{prop}\label{prop:hyperuniformity}
	Let $\eta$ be a stationary point process on $\mathbb{R}^d$ with a truncated pair correlation measure $\beta$ s.t.\ $|\beta|(\mathbb{R}^d)<\infty$. The following are equivalent:
	\begin{enumerate}[(a)]
		\item The point process $\eta$ is hyperuniform, i.e.~\eqref{eq:HU} holds for boxes, balls or arbitrary convex sets. 
		\item It holds that $1+\beta(\mathbb{R}^d)=0$.
	\end{enumerate}
\end{prop}
Note that the definition of hyperuniformity in \eqref{eq:HU} is sometimes stated using balls instead of boxes; however, Proposition \ref{prop:hyperuniformity} verifies that in the setting of this work, both definitions are equivalent.

\begin{rmk}\label{rem:sum=0} Let us comment on the total sum of covariances. For instance, in the integrable case $\int|y||\beta|(\mathrm{d} y)<\infty$, we have $\sum_{z\in\mathbb{Z}^d}\mathrm{cov}(z)= 1-\sum_{z\in\mathbb{Z}^d:|z|=1} \frac{1}{2d}=0$. Similarly, in the non-integrable case of Theorem \ref{thm:Cov_non_int}, it is also easy to verify that $\sum_{z\in\mathbb{Z}}\mathrm{cov}(z)=0$.

In general, $\sum_{z\in\mathbb{Z}}\mathrm{cov}(z)=0$ can be explained by formally setting $g=1$ in \eqref{eq:400}. More rigorously, in the case of $|\beta|(\mathbb{R}^d)<\infty$ and hyperuniformity $\beta(\mathbb{R}^d)=-1$,  we set $g=\ind_{[-m_n,m_n]^d}$ for $m_n\gg n$ sufficiently large such that
\begin{align*}
 \Big|\mathrm{Cov}\big( \eta(\Lambda_n),\eta([-m_n,m_n]^d)\big)
\Big| &=\Big|\lambda_d(\Lambda_n ) + \int_{\Lambda_n}\beta([-m_n,m_n]^d-x)\lambda_d(\mathrm{d}x)\Big|\\
 &\le |\beta|\big(([-m_n+n,m_n-n]^d)^c\big)\cdot\lambda_d(\Lambda_n)\to 0.
\end{align*}
Assuming existence of $\mathrm{cov}(z)$ and uniform summability in $z$, we get as $n\to\infty$
	\begin{align*}		\lim_{n\to\infty}\sum_{z\in\mathbb{Z}^d\cap [-m_n/n,m_n/n]^d}\mathrm{cov}(z)&=\lim_{n\to\infty} \sum_{z\in\mathbb{Z}^d\cap [-m_n/n,m_n/n]^d}\frac{\mathrm{Cov}\big( \eta(\Lambda_n),\eta(\Lambda_n(nz))\big)}{\mathrm{Var}\big(\eta(\Lambda_n)\big)}\\
    &=\lim_{n\to\infty} \frac{\mathrm{Cov}\big( \eta(\Lambda_n),\eta([-m_n,m_n]^d)\big)}{\mathrm{Var}\big(\eta(\Lambda_n)\big)}=0.
	\end{align*} 

Conversely, if the series $\sum_{z\in\mathbb{Z}^d}\mathrm{cov}(z)$ does not absolutely converge, then \emph{both} positive and negative parts $\sum_{z\in\mathbb{Z}^d}\mathrm{cov}(z)_+=\sum_{z\in\mathbb{Z}^d}\mathrm{cov}(z)_-=\infty$ must diverge. Indeed, assume to the contrary  that one of the two is finite, then again by splitting $\Lambda_{n^2}$ into boxes $\Lambda_n$, 
\begin{align}
1&=\sum_{v,w\in\mathbb{Z}^d\cap \Lambda_{n}}\frac{\mathrm{Cov}\big(\eta(\Lambda_n(nv)),\eta(\Lambda_n(nw))\big)}{\mathrm{Var}(\eta(\Lambda_{n^2}))}\nonumber\\
&= \frac{\frac 1{ n^d} \mathrm{Var}(\eta(\Lambda_{n}))}{\frac 1 {n^{2d}}\mathrm{Var}(\eta(\Lambda_{n^2}))} \sum_{z\in\mathbb{Z}^d\cap[-n,n]^d}\frac{\mathrm{Cov}\big(\eta(\Lambda_n),\eta(\Lambda_n(nz))\big)}{\mathrm{Var}(\eta(\Lambda_{n}))} \prod_{i=1}^d\left(1-\frac{|z_i|}n\right)_+,\label{eq:MK}
\end{align}
where counting all $z=v-w\in[-n,n]^d$ for $v,w\in\mathbb{Z}^d\cap \Lambda_n$ leads to the pyramidal function $n^d\prod_{i=1}^d (1-\frac{|z_i|}n)_+\in[0,1]$ (which we are going to use again in our main proofs below). The first factor in the last line is bounded from below infinitely often (generically, it diverges to $\infty$), the first term in the sum becomes $\mathrm{cov}(z)$ and the pyramidal function converges pointwise to $1$. Hence, (formally, up to uniform summability of $\mathrm{Cov}\big(\eta(\Lambda_n),\eta(\Lambda_n(nz))\big)/{\mathrm{Var}(\eta(\Lambda_{n}))}$), we observe that \eqref{eq:MK} has the same limit value as $\sum_{z\in\mathbb{Z}^d}\mathrm{cov}(z)= \pm\infty$ which is a contradiction.
\end{rmk}

The above remark shows that there exists some $\mathrm{cov}(z)<0$ since $\mathrm{cov}(0)=1>0$.
Typically, we expect $\mathrm{cov}(z)\le 0$ for any boxes with disjoint interior (i.e.~for $z\notin [-1,1]^d$) and in some more general setting. The following remark gives a glimpse on this general question without assuming $|\beta|(\mathbb{R}^d)<\infty$.

\begin{rmk}
Let $\eta$ be a hyperuniform point process and assume that $\mathrm{cov}(z)$ exists for all $z\in\mathbb{Z}^d$. Then there exists $z\in\mathbb{Z}^d\setminus \{0\}$ such that $\mathrm{cov}(z)<0$. 
	
Indeed, assume, to the contrary, that $\mathrm{cov}(z)\ge 0$ for all $z\in\mathbb{Z}^d\setminus \{0\}$.
	Let $c_n:=n^d/\mathrm{Var}(\eta(\Lambda_n))$ and note that $c_n=\mathcal O (n)$ goes to infinity for $n\to\infty$. We partition $\Lambda_n$ into $m=\lceil n/\sqrt {c_n}\rceil^d$ smaller boxes $\Lambda_{n/\lceil n/\sqrt {c_n}\rceil}(z_i)$ for $z_i\in n/\lceil n/\sqrt {c_n}\rceil\cdot \mathbb Z^d\cap \Lambda_n$. In particular, $\eta(\Lambda_n)=\sum_{i=1}^m\eta(\Lambda_{n/\lceil n/\sqrt {c_n}\rceil}(z_i))$. Hence,
	\begin{align*}
		1&= \limsup_{n\to\infty}\frac{c_n\mathrm{Var}(\eta(\Lambda_n))}{n^d}\\
        &=\limsup_{n\to\infty}\frac{c_n}{n^d}\Big( \sum_{i=1}^m\mathrm{Var}(\eta(\Lambda_{n/\lceil n/\sqrt {c_n}\rceil}(z_i)))+2\sum_{i<j}^m \mathrm{Cov}(\eta(\Lambda_{n/\lceil n/\sqrt {c_n}\rceil}(z_i)),\eta(\Lambda_{n/\lceil n/\sqrt {c_n}\rceil}(z_j))) \Big)\\
		&\ge\limsup_{n\to\infty} \frac{c_n}{n^d} m \frac{n^d/\lceil n/\sqrt {c_n}\rceil^d}{c_{n/\lceil n/\sqrt {c_n}\rceil}}\\
        &=\limsup_{n\to\infty}\frac{c_n}{c_{n/\lceil n/\sqrt {c_n}\rceil}}\\
        &\ge  \limsup_{n\to\infty}\frac{\sqrt{ c_n}}{C}=\infty,
	\end{align*}
    where the last two lines follow from $c_{n/\lceil n/\sqrt {c_n}\rceil}\sim c_{\sqrt c_n}\le C \sqrt {c_n}$ for some constant $C>0$ and $n$ sufficiently large.
	 This is a contradiction, proving the claim.
\end{rmk}

\section{The integrable case}\label{sec:int}

In this section, we show that for hyperuniform point processes satisfying the integrability assumption 
\begin{align}\label{eq:IA}\tag{IA}
	\int|y||\beta|(\mathrm{d} y)<\infty.
\end{align}
As announced in the introduction, we assume $\beta$ to be \emph{box-symmetric}, that is, symmetric with respect to its coordinates, i.e.~$\beta(\mathrm{d} x_1,\dots,\mathrm{d} x_d)=\beta(\mathrm{d} x_{\sigma(1)},\dots, \mathrm{d} x_{\sigma(d)})$ for all permutations $\sigma$, and symmetric with respect to reflection of individual coordinates, i.e.~$\beta(\mathrm{d} x_1,\dots,\mathrm{d} x_d)=\beta(\mathrm{d} (-x_1),\dots, \mathrm{d} x_d)$. For example, if $\beta$ is rotationally invariant, then it is box-symmetric.\\
As claimed in Theorem \ref{thm:Cov_integrable}, we will show that the covariance structure is given by 
\begin{align*}
	\mathrm{cov}(z)=\begin{cases}-\frac 1 {2d}\lambda_{d-1}(\partial \Lambda_1\cap \partial \Lambda_1(z)),\quad &\text{if } \mathring\Lambda_1\cap \mathring\Lambda_1(z)=\emptyset,\\
		\frac 1 {2d}\lambda_{d-1}(\partial \Lambda_1\cap \partial \Lambda_1(z)),\quad &\text{if }\mathring\Lambda_1\cap \mathring\Lambda_1(z)\neq\emptyset,
	\end{cases}
\end{align*}
where $\partial \Lambda_1(z)$ is the boundary of $\Lambda_1(z)$ and $\mathring\Lambda_1(z)$ is the interior of $\Lambda_1(z)$.
Prominent examples that meet the assumption \eqref{eq:IA} include the Ginibre process on $\mathbb{R}^2\simeq\mathbb{C}$ with $\beta(\mathrm{d} z)=-\frac 1 \pi \exp({-|z|^2})\mathrm{d} z$, GAF-zeros (for $\beta$ see for instance \cite[Equation (2.11)]{YogeshKrish}) and determinantal point processes with $\beta$ depending on the kernels.

\begin{proof}[Proof of Theorem \ref{thm:Cov_integrable}]
	By symmetry, it is sufficient to consider points $z=(z_i)_{i=1,\dots,d}$ located in the first quadrant $\Lambda_\infty:=[0,\infty)^d$.
	We begin with the formula \eqref{eq:Cov}
	\begin{align}\label{eq:Cov1}
		\mathsf{C}_n(z):=\frac {\mathrm{Cov}\big( \eta(\Lambda_n),\eta(\Lambda_n(nz))\big)}{n^d}=\lambda_d(\Lambda_1\cap \Lambda_1(z))+\frac 1 {n^d}\int_{\Lambda_n}\beta(\Lambda_n(nz-x))\lambda_d(\mathrm{d} x).
	\end{align}
	We swap the order of integration in the second term and obtain
	\begin{align*}
		\frac 1 {n^d}\int_{\Lambda_n}\beta(\Lambda_n(nz-x))\lambda_d(\mathrm{d} x)=\int \lambda_d(\Lambda_1\cap \Lambda_1(z-y/n))\beta(\mathrm{d} y).
	\end{align*}
	These boxes have larger overlap whenever the coordinates $y_i/n$ are close to $z_i$ and no overlap, when $|z_i-y_i/n|>1$. More precisely, 
	$$\lambda_d(\Lambda_1\cap \Lambda_1(z-y/n))=\prod_{i=1}^d\big(1-\big|z_i-\frac{y_i}n\big\rvert\big)_+,$$ where $x_+=\max(0,x)$. Thus, equation \eqref{eq:Cov1} can be rephrased as
	\begin{align}\label{eq:Cov2}
		\mathsf{C}_n(z)=\prod_{i=1}^d(1-z_i)_++\int\prod_{i=1}^d\big(1-\big|z_i-\frac{y_i}n\big\rvert\big)_+\beta(\mathrm{d} y).
	\end{align}
	We expect this term to be of order $n^{-1}$ as the variance, and will now analyse it in cases ordered by increasing complexity.\\
	\textsc{Case 1, separated boxes:} Assume that $z_j>1$ for some coordinate $z_j$ of $z$, then the first term in \eqref{eq:Cov2} vanishes and the second can be bounded by
	\begin{align*}
		|\mathsf{C}_n(z)|\le \int \frac {y_j}n \ind_{[(z_j-1)n,\infty)}(y_j) |\beta|(\mathrm{d} y)=o(n^{-1}).
	\end{align*}
	The first inequality follows from $(1-|z_j-y_j/n|)_+\le y_j/n $ for $y_j\ge (z_j-1)n$ 
	(and zero, else) and the second bound follows from our integrability assumption $\int|y||\beta|(\mathrm{d} y)<\infty$. 
	Below, in Case 4, we will verify that $\mathrm{Var}\big(\eta(\Lambda_n)\big)$ is of order $n^{d-1}$ from which it follows that $\mathrm{cov}(z)=0$ 
	as claimed.\\
	\textsc{Case 2, at the edge:}
	Assume that $z_j=z_k=1$ for at least two coordinates of $z$, then the first term in \eqref{eq:Cov2} vanishes, and the second can again be bounded by
	\begin{align*}
		|\mathsf{C}_n(z)| &\le \int \frac {y_jy_k}{n^2} \ind_{[0,2n]}(y_j)\ind_{[0,2n]}(y_k) |\beta|(\mathrm{d} y)\\
		&=\frac 1{n^2} \int y_j \big(y_k\ind_{[0,\sqrt n]}(y_k)+y_k\ind_{[\sqrt n,2n]}(y_k)\big) |\beta|(\mathrm{d} y)\\
		& \le\frac 1{n^2} \int |y| \big(\sqrt n+2n\ind_{[\sqrt n,\infty)}(y_k)\big) |\beta|(\mathrm{d} y)=o(n^{-1}).
	\end{align*}
	\\
	\textsc{Case 3, face to face:}
	Assume that exactly one coordinate satisfies $z_j=1$ and all others satisfy $0\le z_i<1$. Again, the first term in \eqref{eq:Cov2} vanishes and the remainder is
	\begin{align*}
		\mathsf{C}_n(z)=\int\frac{y_j}n\ind_{[0,n]}(y_j)\prod_{i\neq j}\big(1-\big|z_i-\frac{y_i}n\big\rvert\big)_+ \beta(\mathrm{d} y)+o(n^{-1}),
	\end{align*}
	where we used that the integral $\int \frac {y_j}n \ind_{[n,2n]}(y_j)|\beta|(\mathrm{d} y)$ is negligible as argued in Case 1. Observe that each factor in the product is
	\begin{align}\label{eq:produkt}
		\big(1-\big|z_i-\frac{y_i}n\big\rvert\big)_+=\big(1-z_i+\frac{y_i}n\big)\ind_{[n(z_i-1),nz_i]}(y_i)+\big(1+z_i-\frac{y_i}n\big)\ind_{[nz_i,n(z_i+1)]}(y_i).
	\end{align}
	When expanding the product, all terms containing at least one $\frac{y_i}n$ are negligible, as they are multiplied by another $y_j$ analogous to Case 2. Thus,
	$$\mathsf{C}_n(z)=\int\frac{y_j}n\ind_{[0,n]}(y_j)\prod_{i\neq j}\Big((1-z_i) \ind_{[n(z_i-1),nz_i]}(y_i)+(1+z_i)\ind_{[nz_i,n(z_i+1)]}(y_i)\Big)\beta(\mathrm{d} y)+o(n^{-1}).$$
	Moreover, we may extend the area of integration to $y_i\in\mathbb{R}$, since the remaining parts $(-\infty,n(z_i-1)]$ and $[n(z_i+1),\infty)$ are negligible as in Case 1. Indeed, for any $c>0$
	\begin{align}\label{eq:nocheineabschaetzung}
		\int\frac{y_j}n\ind_{[0,n]}(y_j) \ind_{[cn,\infty)}(y_i)\beta(\mathrm{d} y)\le   \frac 1 n \int|y| \ind_{[cn,\infty)}(y_i)\beta(\mathrm{d} y)=o(n^{-1}).
	\end{align}
	
	Lastly, note that if $z_i=0$, then its sign does not matter. Consequently, we arrive at
	\begin{align}
		\mathsf{C}_n(z)&=\int\frac{y_j}n\ind_{[0,n]}(y_j)\prod_{i\neq j}\Big((1-z_i) \ind_{(-\infty,nz_i]}(y_i)+(1+z_i)\ind_{[nz_i,\infty)}(y_i)\Big)\beta(\mathrm{d} y)+o(n^{-1})\nonumber\\
		&=\int\frac{y_j}n\ind_{[0,n]}(y_j)\prod_{i\neq j}\big(1-z_i\big) \beta(\mathrm{d} y)+o(n^{-1})\nonumber\\
		&=\frac {\prod_{i\neq j}\big(1-z_i\big)} {2n}\int|y_1| \beta(\mathrm{d} y)+o(n^{-1})\label{eq:Cov3}
	\end{align}
	by symmetry. Observe that $\prod_{i\neq j}\big(1-z_i\big)=\lambda_{d-1}(\partial \Lambda_1\cap \partial \Lambda_1(z))$ as in the claim, and it remains to identify the integral in the variance term of the denominator.
	\\
	\textsc{Case 4, overlapping bodies:} Assume that $0\leq z_i<1$, denote by $k=0,\dots,d$ the number of coordinates $z_i=0$ and by symmetry we suppose these to be the first $k$ coordinates $z_1,\dots,z_k=0$. Taking into account the first summand of \eqref{eq:Cov2}, we consider
	\begin{align*}
		\mathsf{C}_n(z)=\prod_{i=k+1}^d(1-z_i)+\int\prod_{j=1}^k\big(1-\frac{|y_j|}n\big)_+\prod_{i=k+1}^d\big(1-\big|z_i-\frac{y_i}n\big\rvert\big)_+ \beta(\mathrm{d} y).
	\end{align*}
	Note again that, similar to Case 1,
	$$\int \ind_{[nz_i,\infty)}(y_i)|\beta |(\mathrm{d} y)\le \int \ind_{[nz_i,\infty)}(y_i)\frac {|y|}  {z_in} |\beta|(\mathrm{d} y)=o(n^{-1})$$
	for any $z_i>0$ by Markov's inequality. Together with \eqref{eq:produkt} and \eqref{eq:nocheineabschaetzung}, it follows
	\begin{align*}
		\mathsf{C}_n(z)=&\prod_{i=k+1}^d(1-z_i)+\int\prod_{j=1}^k\big(1-\frac{|y_j|}n\big)_+\prod_{i=k+1}^d\big(1-z_i+\frac{y_i}n \big)\beta(\mathrm{d} y)+o(n^{-1}).
	\end{align*}
	By box-symmetry of $\beta$, it moreover holds
	\begin{align*}
		\int y_i\prod_{j=1}^k\big(1-\frac{|y_j|}n\big)_+\beta(\mathrm{d} y)=0.
	\end{align*}
	By expanding the products and using again the argument of Case 2, we end up with
	\begin{align}
		\mathsf{C}_n(z)&=\prod_{i=k+1}^d(1-z_i)+\int\prod_{j=1}^k\big(1-\frac{|y_j|}n\big)_+\prod_{i=k+1}^d\big(1-z_i\big)\beta(\mathrm{d} y)+o(n^{-1})\nonumber\\
		&=\prod_{i=k+1}^d(1-z_i)\Big(1+\int\big(1-\sum_{j=1}^k \frac{|y_j|}n\big) \beta(\mathrm{d} y)\Big) +o(n^{-1})\nonumber\\
		&=-\prod_{i=k+1}^d(1-z_i)k\int \frac{|y_1|}n \beta(\mathrm{d} y) +o(n^{-1}).\label{eq:Cov4}
	\end{align}
	Finally, choosing $z=0$ corresponding to $k=d$, we obtain
	\begin{align}\label{eq:Var2}
		\mathrm{Var}\big(\eta(\Lambda_n)\big)\sim -n^{d-1}d\int|y_1| \beta(\mathrm{d} y).\end{align}
	\textsc{Conclusion:} Combining \eqref{eq:Var2} with \eqref{eq:Cov3}, we obtain the first claimed equation in Theorem \ref{thm:Cov_integrable}. The second equation follows from \eqref{eq:Cov4} and counting $2k$ overlapping faces of area $\prod_{i=k+1}^d(1-z_i)$ each. Ultimately, if boxes overlap without touching faces, we have $k=0$ in \eqref{eq:Cov4} and then the limit vanishes.
\end{proof}

\begin{rmk} \label{rem:rectangles_integrable}
    In the preceding theorem, and throughout the remainder of the paper, we focus on boxes with equal side length. However, the result  of Theorem \ref{thm:Cov_integrable} can be extended to rectangular boxes\footnote{Equivalently, the result can be extended to $\beta$ that is not invariant with respect to permutation of the coordinates, simply by rescaling the coordinates by $b_i$.} of the form $\tilde \Lambda_n^b(nz)=nz+\times_{i=1}^d[0,nb_i]$ for $(b_1,\ldots, b_d)=b,z\in\mathbb{R}^d$, by following exactly the same line of argumentation as in the proof of Theorem \ref{thm:Cov_integrable}. Without symmetry arguments, all resulting formulas then depend on $b$ via the amount of the boundary overlap,  for instance \eqref{eq:Cov4} for $z_i=0$ for $i=1,\dots, k$ and $0<z_i<b_i$ for $i=k+1\dots, d$ becomes 
    $$\mathsf{C}_n(z)=-\prod_{i=1}^d(b_i-z_i)\sum_{j=1}^k\frac 1 {b_j}\int \frac{|y_1|}{n}\beta (\mathrm d y)+o(n^{-1})=-\frac {\lambda_{d-1}(\partial \tilde\Lambda_1^b\cap \partial \tilde\Lambda_1^b(z))}2 \int \frac{|y_1|}{n}\beta (\mathrm d y)+o(n^{-1}).$$ Consequently, for any $z\in\Lambda_\infty$ the relative covariance becomes the relative boundary overlap
    \begin{align*}
        \mathrm{cov}^b(z):=
        \lim_{n\to\infty}\frac{\mathrm{Cov}\big( \eta(\tilde\Lambda_n^b),\eta(\tilde\Lambda_n^b(nz))\big)}{\mathrm{Var}\big(\eta(\tilde\Lambda_n^b))\big)}=
        \begin{cases}
        -\frac {\lambda_{d-1}(\partial \tilde\Lambda_1^b\cap \partial \tilde\Lambda_1^b(z))} {\lambda_{d-1}(\partial \tilde\Lambda_1^b)},\quad &\text{if } \mathring{\tilde{\Lambda}}_1^b\cap \mathring{\tilde\Lambda}_1^b(z)=\emptyset,\\
		\ \frac {\lambda_{d-1}(\partial {\tilde\Lambda_1^b}\cap \partial \tilde\Lambda_1^b(z))} {\lambda_{d-1}(\partial \tilde\Lambda_1^b)},\quad &\text{if }\mathring{\tilde{\Lambda}}_1^b\cap \mathring{\tilde\Lambda}_1^b(z)\neq\emptyset.
	\end{cases}
    \end{align*}
    Since this extension provides no additional conceptual insight and would make the proof unnecessarily technical, we omit the details.
\end{rmk}

\section{The non-integrable case} \label{sec:nonint}
The integrability assumption \eqref{eq:IA} is significantly stronger than the condition $ |\beta|(\mathbb{R}^d) < \infty $, which is required for the hyperuniformity criterion in Proposition \ref{prop:hyperuniformity}. In this section, we study the covariance structures of hyperuniform point processes that do not satisfy \eqref{eq:IA} but $ |\beta|(\mathbb{R}^d) < \infty $. We begin with the proof of Theorem \ref{thm:Cov_non_int}, which gives a complete description of the covariance structure for such point processes in dimension $d=1$. Thereafter, we extend our study to $ d = 2 $ (Theorem \ref{thm:Cov_non_int_2}) as a model case for higher dimensions. At the end of this section, we construct and discuss several examples of point processes that meet the assumptions of Theorem \ref{thm:Cov_non_int} and Theorem \ref{thm:Cov_non_int_2}.
\subsection{The 1-dimensional case}
In this section, we prove Theorem \ref{thm:Cov_non_int}, which gives a full characterization of the covariance structure beyond the integrability assumption \eqref{eq:IA} for $d=1$. To this end, define $F(x):=\beta([x,\infty))$ and consider the point processes for which $H(y):=-\int_0^y F(x)\mathrm{d}x$ is regularly varying with parameter $a\in[0,1]$, i.e.
\begin{align}\label{eq:regvarr}
	\lim_{n\to\infty} \frac{H(xn)}{H(n)}=x^a,\quad \text{ for all } x\in\mathbb{R}.
\end{align} 
Among other things, we prove that the covariance structure is given by 
\begin{align}\label{eq:Cov_non_int}
	\mathrm{cov}(z)=
	\frac{|z-1|^{a}}{2}+ \frac{|z+1|^{a}}{2}-|z|^{a}, \quad z\in\mathbb{R},
\end{align}
for $a\in (0,1]$ or $\mathrm{cov}(z)=\delta_0(z)-\tfrac 1 2 \delta_{-1}(z)-\tfrac 12\delta_{+1}(z)$ for $a=0$.

\begin{proof}[Proof of Theorem \ref{thm:Cov_non_int}]
	We begin with a preparation, along the same route as in \eqref{eq:Cov1}, to obtain
	\begin{align*}
		\mathrm{Cov}(\eta(\Lambda_n), \eta(\Lambda_n(nz)))&=\lambda_1([0,n]\cap [zn, n(z+1)])+ \int_{\Lambda_n}\beta([zn-x,n(z+1)-x]) \mathrm{d}x\nonumber\\
		&=n(1-|z|)_++ \int_0^n F(zn-x)-F(n(z+1)-x) \mathrm{d}x\nonumber\\
		&= n(1-|z|)_+ + \int^{nz}_{n(z-1)}F(x) \mathrm{d}x-\int^{n(z+1)}_{nz}F(x) \mathrm{d}x\\
		&= n(1-|z|)_+-\int^{n(z-1)}_{0}F(x) \mathrm{d}x + 2\int^{nz}_{0}F(x) \mathrm{d}x-\int^{n(z+1)}_{0}F(x) \mathrm{d}x.
	\end{align*}
	Combining hyperuniformity and the symmetry of $\beta$ yields  
	\begin{align}\label{eq:hyp+sym}
		-1=F(-x)+F(x),\ x\in\mathbb{R}.
	\end{align}
	Therefore,  if $z<0,$ then  
	\begin{align*}
		& -\int^{n(z-1)}_{0}F(x) \mathrm{d}x + 2\int^{nz}_{0}F(x) \mathrm{d}x-\int^{n(z+1)}_{0}F(x)\mathrm{d}x\\&\quad= \int^{n(z-1)}_{0}1+F(-x) \mathrm{d}x - 2\int^{nz}_{0}1+F(-x) \mathrm{d}x-\int^{n(z+1)}_{0}1+F(-x)\mathrm{d}x\\
		&\quad= -\int^{n(|z|-1)}_{0}F(x) \mathrm{d}x + 2\int^{n|z|}_{0}F(x) \mathrm{d}x-\int^{n(|z|+1)}_{0}F(x)\mathrm{d}x.
	\end{align*}
	Further, if $|z|<1$, then the first two terms can be rewritten by using again \eqref{eq:hyp+sym} as
	$$n(1-|z|)_+-\int^{n(|z|-1)}_{0}F(x) \mathrm{d}x=\int_{n(|z|-1)}^01+F(x)\mathrm{d}x=-\int_{n(|z|-1)}^0F(-x)\mathrm{d}x=-\int_0^{n||z|-1|}F(x)\mathrm{d}x.$$
	All in all, we obtain
	\begin{align}\label{eq:Cov_non_int2}
		\mathrm{Cov}(\eta(\Lambda_n), \eta(\Lambda_n(nz)))= 2\int^{n|z|}_{0}F(x) \mathrm{d}x-\int_0^{n||z|-1|}F(x)\mathrm{d}x-\int^{n(|z|+1)}_{0}F(x) \mathrm{d}x.
	\end{align}
	Setting $z=0$, the equality in claim (i) follows
	\begin{align}\label{eq:var}
		\mathrm{Var}(\eta(\Lambda_n))= -2\int_{0}^n F(x) \mathrm{d}x.
	\end{align}
	Now, we assume (i), that is integrals $H(y)=-\int_0^yF(x)\mathrm d x$ to be regularly varying with parameter $a\in[0,1]$. Then, we conclude (ii): If $a\in(0,1]$, then
	\begin{align*}
		\mathrm{cov}(z)=\lim_{n\to\infty}\frac{\mathrm{Cov}(\eta(\Lambda_n), \eta(\Lambda_n(zn)))}{\mathrm{Var}(\eta(\Lambda_n))}=
		\frac{||z|-1|^{a}}{2}+ \frac{(|z|+1)^{a}}{2}-|z|^{a}=
		\frac{|z-1|^{a}}{2}+ \frac{|z+1|^{a}}{2}-|z|^{a}.
	\end{align*}
	For the slowly varying case $a=0$, we obtain $\mathrm{cov}(0)=1$, $ {\mathrm{cov}(1)=\mathrm{cov}(-1)=-1/2}$, and $\mathrm{cov}(z)=0$ for $z\notin \{-1,0,1\}$.
	
	The implication (ii) to (iii) is trivial.
	
	Assume (iii), that the limit
	$$ \mathrm{cov}(z)= \lim_{n\to\infty}\frac{\mathrm{Cov}(\eta(\Lambda_n), \eta(\Lambda_n(zn)))}{\mathrm{Var}(\eta(\Lambda_n))} =\lim_{n\to\infty}\frac {-2H(nz)+H(n|z-1|)+H(n(z+1))}{2H(n)}<\infty$$
	exists for $z=1, z=2$. Since $H(0)=0$, it follows from $z=1$ that $\lim_{n\to\infty}\frac{H(2n)}{H(n)}<\infty$ and consequently, it follows from $z=2$ that  $\lim_{n\to\infty}\frac{H(3n)}{H(n)}<\infty$. Our additional assumption $F(x)\le 0$ (or $\ge 0$) for $x$ sufficiently large implies that $\lim_{n\to\infty}\frac{H((z+1)n)}{H(n)}\ge 1$  for $z=1,2$. Hence, we can apply Karamata's theorem \cite[Theorem 1.10.2]{regvar}: If $H$ is positive and monotone on some neighborhood of infinity and $\lim_{n\to\infty}\frac{H(\gamma n)}{H(n)}\in(0,\infty)$ for two values $\gamma_1,\gamma_2$ such that $\log(\gamma_1)/\log(\gamma_2)$ is finite and irrational, then $H$ is regularly varying. Indeed, choosing $\gamma_1=2, \gamma_2=3$ and noting that $\log(2)/\log(3)$ is irrational, it follows that $H$ is regularly varying with some parameter $a\in\mathbb{R}$. 
	
	It remains to argue that $a\in[0,1]$. Indeed $a<0$ in \eqref{eq:var} would imply $\mathrm{Var}(\eta(\Lambda_n))\to 0$ as $n\to\infty$, which contradicts stationarity. On the other hand, $a>1$ contradicts $F(x)\to 0$ as $x\to\infty$ due to $|\beta|(\mathbb{R}^d)$ being finite. Hence, we proved (i).\end{proof}
\subsection{The higher dimensional case}\label{sec:highdim}
We now turn to the question how to extend Theorem \ref{thm:Cov_non_int} to higher dimensions $d>1$. It turns out that the geometry (in particular inclusion-exclusion principles) in $\mathbb{R}^d$ becomes increasingly involved and so does the representation of $\mathrm{cov}(z)$. Hence, we consider the case $ d = 2 $ as a model case, which can be extended to higher dimensions using similar methods.
\\

To this end, let us recall the notion of multivariate regularly varying functions, see \cite[\S 6.1]{Resnick} and \cite[Theorem 1]{HaanOmey}\footnote{We share their point of view \emph{"The situation in $d>1$ is more complicated but not essentially different; for simplicity, we limit ourselves to $\mathbb{R}^2$"}.} for more information. Let $\mathbf 1:=(1,1)$ and define the infinite cone  starting at $z\in\mathbb{R}^d$ by $\Lambda_\infty(z):=[z_1,\infty)\times[z_2,\infty)$. We call a function $f:\Lambda_\infty(z)\to [0,\infty)$ multivariate $(a,g)$-regularly varying, if for some $a\in\mathbb{R}$ and $g:\mathbb S^{1}\to (0,\infty)$, we have
\begin{align}\label{eq:mult_reg_var}
	\lim_{n\to\infty} \frac{f(nx)}{f(n\mathbf 1)}=\|x\|^{a}g(x/\|x\|),\qquad \text{for all } x\in \Lambda_\infty(z).
\end{align}

\begin{satz}\label{thm:Cov_non_int_2}
	Let $\eta$ be a stationary hyperuniform point process on $\mathbb{R}^2$ with truncated pair correlation measure $\beta.$ Suppose that $\beta$ is box-symmetric (with respect to reflection and swapping of the axes), that ${|\beta|(\mathbb{R}^2)<\infty}$, and that $F(x):=\beta(\Lambda_\infty(x))$ is non-positive (or, non-negative) for $\|x\|$ sufficiently large. 
	Assume further that both $H_\pm(y_1,y_2):= \int_0^{y_1}\int_0^{y_2}F(\pm x_1,x_2)\mathrm{d}x_1\mathrm{d}x_2$ are multivariate $(a_\pm,g_\pm)$-regularly varying, each with $a_\pm\in[0,1]$ and such that $g_-(x,y)=g_-(y,x)$. Moreover assume $K_+:=\lim_{n\to\infty}H_+(n,n)/H_-(n,n)$ exists and set $K_-=-1$.
	Then, $\mathrm{cov}(z)$ exists and can be expressed similar to \eqref{eq:Cov_non_int}.

	More precisely, denoting $|z|:=(|z_1|,|z_2|)$ and $\tau(z)=\min(\mathrm{sign}(z_1),\mathrm{sign}(z_2))$ (i.e. $\tau(z)=-$ if and only if at least one of $z_1,z_2$ has negative sign), then for $z\in\Lambda_\infty(0)$, we have
	\begin{align*}
		\mathrm{cov}(z)= &  -K_+\|z\|^ag_+\left(\frac{z}{\|z\|}\right)+\sum_{\substack{e\in\mathbb{Z}^2\\ \|e\|=1}}\frac{K_{\tau(z+e)}}{2} \|z+e\|^ag_{\tau(z+e)}\left(\frac{|z+e|}{\|z+e\|}\right)\\
		&-\sum_{v\in\{-1,1\}^2}\frac{K_{\tau(z+v)}}{4}\|z+v\|^ag_{\tau(z+v)}\left(\frac{|z+v|}{\|z+v\|}\right)+ \ind_{\Lambda_1}(z)\frac{1}{2}\|z-\mathbf 1\|^ag_{\tau(z+v)}\left(\frac{|z-\mathbf 1|}{\|z-\mathbf 1\|}\right).
	\end{align*}
\end{satz}
Note that Theorem \ref{thm:Cov_non_int} shows that the mere assumption of regular variation is weak for $d=1$. However, for $d=2$ there are now two different regularly varying functions, since the tail $F(-x_1,x_2)$ may contain much more mass than $F(x_1,x_2)$ whose slowly varying part should be compatible such that $K_+$ exists. It follows automatically that $a_+\le a_-$ and $K_+\le 1$. Depending on the ratio, $\mathrm{cov}(z)$ may vanish non-isotropically, which can be seen from an inspection of the proof (e.g.\ $F(x_1,x_2)=((x_1+1)(x_2+1))^{-a}$, $x_1,x_2>0$, $a\in(0,1)$, implies $a_+=2-2a<a_-=2-a$ such that $K_+=0$ and one may show that $\mathrm{cov}$ vanishes outside of a 'cross' in $\mathbb{R}^2$). 
A tail-function $F$, which avoids such unwanted effects in Theorem \ref{thm:Cov_non_int_2} is e.g. $F(x_1,x_2)=(x_1 +x_2 +1)^{-a}$, $x_1,x_2>0$. A point process having such a truncated pair correlation measure $\beta$ can be constructed as in Example \ref{ex:ex}. Assuming that $\beta$ is not only box-symmetric but even isotropic does neither simplify the statement nor the remaining assumptions of Theorem \ref{thm:Cov_non_int_2}, since such an isotropy does not carry over to infinite boxes like $F(x):=\beta(\Lambda_\infty(x))$, which we shall need during the proof.
\begin{proof}
	By symmetry, it suffices to consider  $z=(z_1,z_2)\in\Lambda_\infty(0).$ Let us denote the coordinate vectors by $e_1=(1,0),e_2=(0,1)$. By formula \eqref{eq:Cov}, splitting the box $\Lambda_n(nz-x)$ and using inclusion and exclusion principle, we obtain that 
	\begin{align}
		&\mathrm{Cov}(\eta(\Lambda_n),\eta(\Lambda_n(nz)))\label{eq:2}=n^2(1-z_1)_+(1-z_2)_+\\
		&+\int_{\Lambda_n(0)} F(nz-x)-F(n(z+e_1)-x)-F(n(z+e_2)-x)+F(n(z+e_1+e_2)-x)\mathrm d x\nonumber
	\end{align}
	For notational convenience, we abbreviate $(\int_A+\int_B)F(x)\lambda_d(\mathrm{d}x)=\int_AF(x)\lambda_d(\mathrm{d}x)+\int_BF(x)\lambda_d(\mathrm{d}x)$ and $\int_0^{-c}=-\int_{-c}^0$, $c>0$, in the following.
	Shifting and splitting the four squares of the integration area into rectangles with a corner at $0$, the above can be reformulated by
	\begin{align}\label{eq:100}
		&\mathrm{Cov}(\eta(\Lambda_n),\eta(\Lambda_n(nz)))
		=n^2(1-z_1)_+(1-z_2)_+\\
		&+\Big(\int_{n(z_1-1)}^{nz_1}\int_{n(z_2-1)}^{nz_2}
		-\int_{nz_1}^{n(z_1+1)}\int_{n(z_2-1)}^{nz_2}
		-\int_{n(z_1-1)}^{nz_1}\int_{nz_2}^{n(z_2+1)}
		+\int_{nz_1}^{n(z_1+1)}\int_{nz_2}^{n(z_2+1)}\Big)
		F(x)\lambda_2(\mathrm{d}x)\nonumber\\
		&=
		n^2(1-z_1)_+(1-z_2)_++\Big(
		4\int_0^{nz_1}\!\!\!\int_0^{nz_2} 
		-2 \int_{0}^{nz_1}\!\!\!\int_{0}^{n(z_2-1)}
		-2\int_{0}^{nz_1}\!\!\!\int_{0}^{n(z_2+1)}
		-2\int_{0}^{n(z_1-1)}\!\!\!\int_{0}^{nz_2}\nonumber\\
		&-2\int_{0}^{n(z_1+1)}\!\!\!\!\int_{0}^{nz_2}
		\!+\!\int_{0}^{n(z_1+1)}\!\!\!\!\int_{0}^{n(z_2+1)}
		\!+\!\int_{0}^{n(z_1+1)}\!\!\!\!\int_{0}^{n(z_2-1)}
		\!+\!\int_{0}^{n(z_1-1)}\!\!\!\!\int_{0}^{n(z_2+1)}
		\!+\!\int_{0}^{n(z_1-1)}\!\!\!\!\int_{0}^{n(z_2-1)}\Big)
		F(x)\lambda_2(\mathrm{d}x)\nonumber
	\end{align}
	Let us continue with the computation of the variance. By hyperuniformity and the symmetry properties of $\beta$ it holds that 
	\begin{align*}
		-1=\beta(\mathbb{R}^2)=F(x_1,x_2)+F(-x_1,x_2)+F(x_1,-x_2)+F(-x_1,-x_2).
	\end{align*}
	If $z_1,z_2<1$, we get that 
	\begin{align}
		\label{eq:3}n^2(1-z_1)_+(1-z_2)_+
		&+\int_{0}^{n(z_1-1)}\!\int_{0}^{n(z_2-1)}F(x)\lambda_2(\mathrm{d}x)= \int^{0}_{n(z_1-1)}\!\int^{0}_{n(z_2-1)}1+F(x)\lambda_2(\mathrm{d}x)\nonumber\\
		&=-\int_{0}^{n|z_1-1|}\int_{0}^{n|z_2-1|}F(x_1,-x_2)+F(-x_1,x_2)+F(x_1,x_2)\mathrm{d}x_1\mathrm{d}x_2.
	\end{align}
	Combining \eqref{eq:2} and \eqref{eq:3}, we obtain that 
	\begin{align}\label{eq:6}
		\mathrm{Var}(\eta(\Lambda_n))&=\mathrm{Cov}(\eta(\Lambda_n),\eta(\Lambda_n))\nonumber\\
		&=\Big( \int_0^{-n}\int_0^{n}
		+\int_0^{n}\int_0^{-n}
		+\int_0^{n}\int_0^{n}
		-\int_0^{n}\int_{-n}^{0}
		-\int_{-n}^{0}\int_0^{n}
		-\int_{0}^{n}\int_0^{n}\Big)F(x)\lambda_2(\mathrm{d}x)\\
		&=-4\int_{0}^{n}\int_0^{n}F(x_1,-x_2)\mathrm{d}x_1\mathrm{d}x_2.\nonumber
	\end{align}
	In order to obtain $\mathrm{cov}(z)$ explicitly, we may divide the covariance \eqref{eq:2} by the variance \eqref{eq:6}, expand the quotient by $K$ if signs do not match, consider the different cases $z_1,z_2<1$, $z_1\geq1$ or $z_2\geq1$, and $z_1,z_2\geq1$ separately, and use the regularly varying assumption. We consider $z=(z_1,z_2)\in\Lambda_\infty((1,1))$, all other cases can be treated completely similarly. 
	Combining equation \eqref{eq:100} and \eqref{eq:6}, we obtain that 
	\begin{align*}
		&\lim_{n\to\infty}\frac{\mathrm{Cov}(\eta(\Lambda_n),\eta(\Lambda_n(nz)))}{\mathrm{Var}(\eta(\Lambda_n))}\\&= \lim_{n\to\infty}\frac{H_+(n,n)}{H_-(n,n)}\Big(\frac{ 
			4 H_+(nz_1,nz_2) 
			-2 H_+(nz_1,n(z_2-1) 
			-2 H_+(nz_1,n(z_2+1))
			-2 H_+(n(z_1-1),nz_2)}{-4H_+(n,n)}\\
		& \quad+\frac{-2 H_+(n(z_1+1),nz_2) +H_+(n(z_1+1),n(z_2+1))  +H_+(n(z_1+1),n(z_2-1))
		}{-4H_+(n,n)}\\
		&\quad+\frac{
			H_+(n(z_1-1),n(z_2+1))+ H_+(n(z_1-1),n(z_2-1))}{-4H_+(n,n)}\Big)\\
		&= K_+\Big(-\|z\|^ag_+\left(\frac{z}{\|z\|}\right)+\frac{1}{2}\|z-e_2\|^ag_+\left(\frac{z-e_2}{\|z+e_2\|}\right)+\frac{1}{2}\|z+e_2\|^ag_+\left(\frac{z+e_2}{\|z+e_2\|}\right) \\
		&\quad +\frac{1}{2}\|z-e_1\|^ag_+\left(\frac{z-e_1}{\|z-e_1\|}\right) +\frac{1}{2}\|z+e_1\|^ag_+\left(\frac{z+e_1}{\|z+e_1\|}\right)- \frac{1}{4}\|z+(1,1)\|^ag_+\left(\frac{z+(1,1)}{\|z+(1,1)\|}\right) \\
		& \quad- \frac{1}{4}\|z+(1,-1)\|^ag_+\left(\frac{z+(1,-1)}{\|z+(1,-1)\|}\right) - \frac{1}{4}\|z+(-1,1)\|^ag_+\left(\frac{z+(-1,1)}{\|z+(-1,1)\|}\right)\\
		&\quad- \frac{1}{4}\|z+(-1,-1)\|^ag_+\left(\frac{z+(-1,-1)}{\|z+(-1,-1)\|}\right)\Big),
	\end{align*}
	which finishes the proof.
\end{proof}
\begin{rmk}\label{rem:rectangles_non-integrable}
  As in Remark \ref{rem:rectangles_integrable},  the result and the proof of Theorem \ref{thm:Cov_non_int_2} can be extended to non-quadratic boxes $\tilde \Lambda_n^b(nz)=nz+[0,nb_1]\times[0,nb_2]$ for $(b_1,b_2)=b,z\in\mathbb{R}^2$ by following exactly the same line of argumentation. In particular, we obtain
    \begin{align*}
        &\mathrm{Cov}(\eta(\tilde\Lambda_n^b),\eta(\tilde\Lambda_n^b(nz)))\\
        =& n^2(b_1-z_1)_+(b_2-z_2)_+\\
        &+\Big(
		4\int_0^{nz_1}\!\!\!\int_0^{nz_2} 
		-2 \int_{0}^{nz_1}\!\!\!\int_{0}^{n(z_2-b_2)}
		-2\int_{0}^{nz_1}\!\!\!\int_{0}^{n(z_2+b_2)}
		-2\int_{0}^{n(z_1-b_1)}\!\!\!\int_{0}^{nz_2}
		-2\int_{0}^{n(z_1+b_1)}\!\!\!\!\int_{0}^{nz_2}\nonumber\\
		&+\!\int_{0}^{n(z_1+b_1)}\!\!\!\!\int_{0}^{n(z_2+b_2)}
		\!+\!\int_{0}^{n(z_1+b_1)}\!\!\!\!\int_{0}^{n(z_2-b_2)}
		\!+\!\int_{0}^{n(z_1-b_1)}\!\!\!\!\int_{0}^{n(z_2+b_2)}
		\!+\!\int_{0}^{n(z_1-b_1)}\!\!\!\!\int_{0}^{n(z_2-b_2)}\Big)
		F(x)\lambda_2(\mathrm{d}x),
    \end{align*}
    and therefore
$\mathrm{Var}(\tilde\eta(\Lambda_n^b))=-4\int_{0}^{b_1n}\int_0^{b_2n}F(x_1,-x_2)\mathrm{d}x_1\mathrm{d}x_2$. 
    Using the assumptions and the notation of Theorem \ref{thm:Cov_non_int_2}, one can finally deduce that 
    \begin{align*}
		\mathrm{cov}^b(z)&= \frac{1}{\|(b_1,b_2)\|^ag_-\left(\frac{(b_1,b_2)}{\|(b_1,b_2)\|}\right)}\Big(   -K_+\|z\|^ag_+\left(\frac{z}{\|z\|}\right)+\sum_{\substack{e\in E}}\frac{K_{\tau(z+e)}}{2} \|z+e\|^ag_{\tau(z+e)}\left(\frac{|z+e|}{\|z+e\|}\right)\\
		&\quad-\sum_{v\in V}\frac{K_{\tau(z+v)}}{4}\|z+v\|^ag_{\tau(z+v)}\left(\frac{|z+v|}{\|z+v\|}\right)+ \ind_{\tilde\Lambda_1^b}(z)\frac{1}{2}\|z-\mathbf 1\|^ag_{\tau(z+v)}\left(\frac{|z-\mathbf 1|}{\|z-\mathbf 1\|}\right)\Big),
	\end{align*}
    with $E=\{(b_1,0),(-b_1,0),(0,b_2),(0,-b_2)\}$ and $V=\{(\pm b_1,\pm b_2)\}.$ Note that this is consistent with Theorem \ref{thm:Cov_non_int_2} by setting $b=\mathbf 1$ and noting that \eqref{eq:mult_reg_var} gives the prefactor $1=\|\mathbf 1\|^ag_- ( {\mathbf 1}/{\|\mathbf 1\|}) $.
\end{rmk}
\subsection{Examples}\label{sec:examples}
In this section, we discuss and construct examples of point processes satisfying the assumptions of Theorem \ref{thm:Cov_non_int} and  Theorem \ref{thm:Cov_non_int_2}.
\begin{bsp}\label{ex:GUE}
	\begin{enumerate}
		\item[(a)]
		A particularly famous example satisfying the assumption of Theorem \ref{thm:Cov_non_int} for $a=0$ is the $\mathsf{sine}_2$ process (arising via eigenvalues of GUE matrices cf.~\cite{AGZ}) with $$\frac{\mathrm{d}\beta}{\mathrm{d} \lambda_1}(x) =-\frac{\sin(\pi x)^2}{(\pi x)^2}.$$ A direct computation shows that $H(y)= (1/ {2\pi^2})\log(y)+\mathcal O (1)$ is slowly varying, i.e.\ regularly varying with parameter $a=0$.    
		\item [(b)] The previous example can be generalized to all hyperuniform determinantal point processes $\eta$ with real kernel $K(x-y)$ in $d=1$ as follows (we refer to \cite{HKPV,Sosh,LRSurvey} for more information). Its Fourier transform satisfies $0\le\hat K\le 1$ and hyperuniformity is equivalent to vanishing spectral measure/structure factor at $0$, i.e.
		$$1-\widehat K^2(0)=\int\hat K (1-\hat K)d\lambda_1=0.$$
		In particular, $K$ must be a Fourier-restriction kernel, and it follows that $\mathrm{Var} (\eta(\Lambda_n))\sim \frac {c}{2\pi^2}\log n$ for the number $c$ of spectral restrictions, see \cite[\S 1.3]{DL} and \cite[Equation (3.15)]{Sosh}. Therefore, all hyperuniform DPPs in $d=1$ correspond to the non-integrable case with $a=0$ in Theorem \ref{thm:Cov_non_int}.
	\end{enumerate}
\end{bsp}

\begin{bsp}
	Riesz gases with parameter  $a = s \in (0,1) $ are expected to belong to the setting of Theorem~\ref{thm:Cov_non_int} and Theorem~\ref{thm:Cov_non_int_2}. In \cite[Theorem 1]{boursier} and \cite[\S VI.C]{Lewin}, it is explained why ${\beta(\mathrm{d}x)\asymp |x|^{-2+s}}$ (whose second antiderivative then corresponds to Theorem \ref{thm:Cov_non_int}) should hold for high temperature, see also \cite{DerDig}. 
	
\end{bsp}
\begin{bsp}\label{ex:ex} 
	In this example, we construct perturbed lattices for which the tails of the truncated pair correlation measure decay at an arbitrary rate, implying that the assumption of Theorem \ref{thm:Cov_non_int} and Theorem \ref{thm:Cov_non_int_2} are satisfied.
	\begin{enumerate}
		\item[(a)] Let $(X_z)_{z\in\mathbb{Z}^d}$ be a family of i.i.d.\ random variables and $U\sim\mathrm{Unif}(\Lambda_1)$ an additional independent random variable. The truncated pair correlation measure of the stationarized i.i.d.\ perturbed lattice $\eta^X:=\sum_{z\in\mathbb{Z}^d}\delta_{z+X_z+U}$ for $X_z\sim h(x)\lambda_d(\mathrm{d}x)$ is given by \begin{align*}
		    \beta(\mathrm{d}x)&=
            \sum_{z\in\mathbb{Z}^d\setminus\{0\}}\int h(y)h(y+x-z)\lambda_d(\mathrm{d}y)\lambda_d(\mathrm{d}x)-\lambda_d(\mathrm{d}x),
		\end{align*}
        see also \cite[Proposition 2.8]{dereudre2024nonhyperuniformityperturbedlattices}. 
        Let $X_z\sim ({1}/{m^d})\ind_{[0,m]^d}\lambda_d(\mathrm{d}x)=\mathrm{Unif}([0,m]^d)(\mathrm{d}x)$. The total variation of the respective truncated pair correlation measure $\beta=\beta_m$ is finite if and only if $m\in\mathbb{N}$. Indeed, if $m\in\mathbb{N}$ then $\sum_{z\in\mathbb{Z}^d\setminus\{0\}}\int h(y)h(y+x-z)\lambda_d(\mathrm{d}y)=1$ for $|x|$ large enough, which precisely is the "cloaking phenomenon" described in \cite{KlattCloaked}.\footnote{More precisely, cloaking refers to the absence of atoms in the spectral measure $\hat\beta$, and by a Fourier-transformation this is implied by the integrability assumption $|\beta|(\mathbb{R})<\infty$. The converse is false in general, but for the i.i.d.\ perturbed lattice, such atoms may only come from the lattice structure ("Bragg-peaks"), and cloaking becomes equivalent to $|\beta|(\mathbb{R})<\infty$ which has been shown in \cite[\S 2, (8), C]{KlattCloaked}}. 
     It follows that $|\beta|(\mathbb{R}^d)=1$. On the other hand, if $m\not\in\mathbb{N}$ then $\sum_{z\in\mathbb{Z}^d\setminus\{0\}}\int h(y)h(y+x-z)\lambda_d(\mathrm{d}y)$ is periodic for large $x$, 
        hence $|\beta|(\mathbb{R}^d)=\infty$.
		\item[(b)]
		By appropriately combining the stationary i.i.d. perturbed lattices constructed in part (a), we can construct (non-ergodic) point processes $\eta$ exhibiting arbitrary tails of the truncated pair correlation measure $\beta$ as follows.  
		Define i.i.d.\ random variables $X_z^m\sim\mathrm{Unif}(\Lambda_m), m\in\mathbb{N},z\in\mathbb{Z}^d$, and let $Y$ be another independent random variable on $\mathbb{N}$ with probability distribution $(p_m)_{m\in\mathbb{N}}$. 
		Then, $\eta:=\sum_{m\in\mathbb{N}}\ind_{\{Y=m\}}\eta^{X^m}$ has $\beta=\sum_m\beta_m p_m$ and $\beta(\{|x|>R\})\asymp \sum_{m>cR} p_m$, which can attain any power $R^{-a}$ by choosing $(p_m)_{m\in\mathbb{N}}$ accordingly.
		\item[(c)] Moreover, one may also choose a different partition of unity by the following procedure, see also \cite{Mastrilli}: Let $(Y_z)_{z\in\mathbb{Z}^d}$ be a sequence of i.i.d.\ random variables with prescribed distribution $\mu$ on $\mathbb{R}^d$. Using the notation of part (b), the process  $\eta^{X^1+Y}=\sum_{z\in\mathbb{Z}^d}\delta_{z+U+X_z^1+Y_z}$ has integrable $\beta$ with tails coinciding with that of $\mu$. In particular, for any $a\in[0,1)$ of our Theorems, one may choose $\mu$ a Pareto distribution with parameter $a$. For $a=1$ one may choose $\mu$ as the log-Cauchy distribution.
	\end{enumerate}
\end{bsp}

\section{The multivariate Central Limit Theorem}\label{sec:cumulant_CLT}
The previous sections and results demonstrate that the microscopic correlations, described by $ \beta $, become negligible at macroscopic scales, and a universal covariance structure emerges. The number statistics $ \eta(\Lambda_n(z)) $ lose memory of the microscopic configuration of the point process $ \eta $, and in the limit $ n \to \infty $, only a ``coarse-grained'' process parametrized by the shift $ z \in \mathbb{R}^d $ remains. In this section, we identify this ``coarse-grained'' process as a Gaussian process by proving a CLT (Theorem \ref{thm:CLT}) for the number statistics of boxes.

A powerful and flexible approach for proving such a CLT is the cumulant method, which we will also adopt in this work. To this end, we define the $k$-th factorial cumulant measure $\gamma^{(k)}$ of a point process $\eta$ by
\begin{align}\label{eq:cumulant_measures}
	\gamma ^{(k)}(\mathrm{d}x_1,\dots,\mathrm{d}x_k):=\sum_{\pi\in\Pi(k)}(-1)^{|\pi|-1}(|\pi|-1)!\prod_{j=1}^{|\pi|}\alpha^{|\pi_j|}(\mathrm{d}x_{\pi_j}),
\end{align}
where $\Pi(k)$ is the set of partitions $\pi=\{\pi_1,\dots,\pi_{|\pi|}\}$ of $\{1,\dots,k\}$ and for  $\pi_j=\{k_1,\dots,k_{|\pi_j|}\}$ we write $\mathrm{d}x_{\pi_j}=\mathrm{d}x_{k_1}\dots \mathrm{d}x_{k_{|\pi_j|}}$. In particular, $\gamma^{(k)}(A^k)=\mathrm{Cum}_k(\eta(A))$ is the classical $k$-th cumulant of the random variable $\eta(A)$, $A\in \mathcal B(\mathbb{R}^d)$. Similar to \eqref{eq:shiftpaircorrelation}, for a stationary point process, we define the \textit{reduced $k$-th cumulant measure} $\gamma^{(k)}_!$ via
\begin{align}\label{eq:truncated_cumulant}
	\int f(x_1,\dots,x_k)\gamma ^{(k)}(\mathrm{d}x_1,\dots,\mathrm{d}x_k)= \int\int f(x_1+y,x_2+y,\dots,x_{k-1}+y,y)\gamma^{(k)}_!(\mathrm{d}x_1,\dots,\mathrm{d}x_{k-1}) \lambda_d(\mathrm{d}y),
\end{align} 
for non-negative measurable functions $f:(\mathbb{R}^{d})^k\to\mathbb{R}_+$.
\begin{dfn}\label{dfn:Brillinger}
	A simple stationary point process $\eta$ is called \textit{Brillinger-mixing} if $\gamma^{(k)}$ exists for all $k\geq 2$ and $|\gamma_!^{(k)}|((\mathbb{R}^d)^{k-1})<\infty$.
\end{dfn}
\begin{rmk}
	Since at least \cite{Ivanoff}, the \emph{Brillinger-mixing} condition has become a standard assumption for point processes in the context of CLTs, see for instance \cite[Theorem 3.9]{Mastrilli}. It is satisfied, for instance, for determinantal point processes with $L^2$ kernel, see \cite{BL16DPP}.
    
    On the one hand it is a strong mixing condition for $\eta$, on the other hand, since $\gamma_!^{(2)}=\beta$, the Brillinger-mixing assumption can be interpreted as a natural extension of our standing assumption $|\beta|(\mathbb{R}^d)<\infty$ to higher order. We do believe that the Brillinger-mixing condition for Theorem \ref{thm:CLT} is technical and may be weakened. However, perturbed stationarized lattices from Example \ref{ex:ex} are neither mixing nor Brillinger-mixing, and the claim of Theorem \ref{thm:CLT} may fail in general; For any fixed $m$ as in Example \ref{ex:ex} (a), we have convergence to $G$ with $a=0$ independent of the choice of $p_m$ in Example \ref{ex:ex} (b). In order to weaken the Brillinger-mixing condition, it may be possible to make use of cancellation effects within the cumulants. This has been utilized in \cite[\S 4]{YogeshKrish}, where differentiability of the test function was crucial and is not applicable in our setting.
    \end{rmk}
    
    The following result enables us to produce explicit examples satisfying the assumptions of Theorem \ref{thm:CLT}, by an i.i.d.~perturbation that transfers its heavy tails.

\begin{exlem}\label{exlem} Let $\eta=\sum_i\delta_{X_i}$ be a stationary point process on $\mathbb{R}^d$ with $k$-th factorial cumulant measure $\gamma^{(k)}$. For any distribution $\mu$ on $\mathbb{R}^d$ define the perturbation $\tilde \eta=\sum_i\delta_{X_i+Y_i}$ for i.i.d.~$Y_i\sim\mu$. Then,
\begin{enumerate}[(a)]
    \item The process $\tilde \eta$ has factorial cumulant measure $\tilde\gamma^{(k)}=\gamma^{(k)}\ast\mu^{\otimes k}$, where $\ast$ is the convolution on $(\mathbb{R}^d)^k$.
    \item The process $\tilde \eta$ has 
$\tilde\gamma^{(k)}_!=\gamma^{(k)}_!\ast\mu^{\Delta}$, where $\mu^\Delta$ is the distribution of $(Y_1-Y_k,\dots,Y_{k-1}-Y_k)$.
    \item If $\eta$  is Brillinger-mixing, then also $\tilde\eta$ is Brillinger-mixing. If for all $k\ge 2$, the measure $\gamma_!^{(k)}$ has a uniform sign, then the converse is also true.
    \item Let $\mu$ on $\mathbb{R}$ have regularly varying density $f$ of order ${a-2}$ for $a\in (0,1)$ and $\eta$ on $\mathbb{R}$ be hyperuniform and Brillinger-mixing with $0\le -\frac{\mathrm d \beta }{\mathrm d x}(x)=o(f(x))$, such as the $\mathsf{sine}_2$ process. Then, $\tilde \eta$ satisfies the assumptions of Theorem \ref{thm:CLT}.
\end{enumerate}
\end{exlem}

Statement $(b)$ for $k=2$ has already  been mentioned in Example \ref{ex:ex} (a). 
In this example, we also saw that the converse statement in $(c)$ may fail without the additional sign assumption, since the perturbation $\mu=\mathrm{Unif}(\Lambda_1)$ of the lattice turned $\beta=\gamma_!^{(2)}$ integrable. We refer to \cite{Andrea} for more details on how $\beta$ changes under perturbation. Even though the stationarized $\mathrm{Unif}(\Lambda_1)$-perturbed lattice has a non-uniformly-signed $\gamma_!^{(3)}$, one may argue that periodic components present in $\gamma_!^{(3)}$ prevent Brillinger-mixing after any perturbation, cf.~\cite{KlattCloaked,LRSurvey}.\footnote{We thank the anonymous referee for raising our attention that led us to clarify this point.} 

\begin{proof} We make use of marking and mapping theorems from \cite[\S 5]{last_penrose_2017}. 
The independent marking $\xi=\sum_i\delta_{(X_i,Y_i)}$ has factorial moment measures $\alpha_\xi^k=\alpha^k\otimes \mu^{\otimes k}$. The perturbed point process $\tilde \eta=T_\#\xi$ is the push-forward of $\xi$ under the mapping $T:(X,Y)\mapsto X+Y$ and thus $T^{\otimes k}=(T,\dots,T)$ pushes the factorial moment measures forward to $\tilde\alpha^k=T^{\otimes k}_\#\alpha_\xi^k=\alpha^k\ast \mu^{\otimes k}$, where $\ast$ is the convolution on $(\mathbb{R}^d)^k$. Therefore, the first claim (a) follows from \eqref{eq:cumulant_measures}, by linearity of $\ast$ and the fact that the product of the disjoint block convolutions combines into a single full $(\mathbb{R}^d)^k$ convolution $$\prod_j\tilde \alpha^{|\pi_j|} =\prod_j\left(\alpha^{|\pi_j|}\ast \mu^{\otimes |\pi_j|}\right)=\left(\prod_j\alpha^{|\pi_j|}\right)\ast \mu^{\otimes k}.$$

The first claim also implies the second claim (b), since \eqref{eq:truncated_cumulant} can be rewritten as $$\gamma^{(k)}(\mathrm d x_1,\dots,\mathrm d x_k)=\gamma_!^{(k)}( \mathrm d (x_1-x_k),\dots ,\mathrm d (x_ {k-1}-x_k))\lambda_d(\mathrm d x_k).$$

The third claim (c) follows from Young's convolution inequality for the total variation norm $$|\tilde\gamma_!^{(k)}|((\mathbb{R}^d)^{k-1})=|\gamma_!^{(k)}\ast\mu^\Delta|((\mathbb{R}^d)^{k-1})\le |\gamma_!^{(k)}|((\mathbb{R}^d)^{k-1})\cdot|\mu^\Delta|((\mathbb{R}^d)^{k-1})= |\gamma_!^{(k)}|((\mathbb{R}^d)^{k-1}).$$
If $\gamma_!^{(k)}$ has a uniform sign, then we have equality as the total variation becomes the total mass and the convolution with a probability measure does not change the total mass.

It remains to prove (d): The i.i.d.~perturbed process $\tilde \eta$ is still Brillinger-mixing by $(c)$ and hyperuniform since $\tilde\beta(\mathbb{R})=\gamma^{(2)}_!\ast\mu^{\Delta} (\mathbb{R})=\beta(\mathbb{R})=-1$ by $(b)$. Importantly, the density $-\frac{\mathrm d \tilde\beta }{\mathrm d x}(x)=-\frac{\mathrm d \beta }{\mathrm d x}\ast f (x) \sim f(x)$ inherits the regularly varying asymptotics as $|x|\to\infty$ of $\mu$ by \cite[Theorem 2.1]{bingham2006regularly} (see also \cite[Lemma 2.34]{Foss}). Therefore, $\tilde\eta$ satisfies the assumption (i) of Theorem \ref{thm:Cov_non_int}.
\end{proof}

We continue with the proof of Theorem \ref{thm:CLT} and show that for hyperuniform point processes in $\mathbb{R}^d,\ d\geq 2$ satisfying the assumptions of Theorem \ref{thm:Cov_integrable}, Theorem \ref{thm:Cov_non_int} or Theorem \ref{thm:Cov_non_int_2} and the Brillinger-mixing assumption, we have weak convergence in finite-dimensional distributions
\begin{align}
	\Bigg(\frac{\eta(\Lambda_n(z))-n^d}{\sqrt{\mathrm{Var}\big(\eta(\Lambda_n)\big)}}\Bigg)_{z\in\mathbb{R}^d}{\underset{n\to\infty}\longrightarrow}\big(G(z)\big)_{z\in\mathbb{R}^d} 
\end{align}
towards a centered stationary Gaussian field $G$ with covariance function $\mathrm{cov}(z)$. Further, we prove that the same statement holds true for point processes on $\mathbb{R}$ with the additional assumption that $\eta$ satisfies \eqref{ieq:Cov_non_int} with $a>0$.
\begin{rmk}\label{rem:rectangles_CLT}
    The result  of Theorem \ref{thm:CLT} can be extended to rectangular boxes of the form $\tilde \Lambda_n^b(nz)=nz+\times_{i=1}^d[0,nb_i]$, for $(b_1,\ldots, b_d)=b,z\in\mathbb{R}^d$ by following the same line of argumentation as in the proof of Theorem \ref{thm:CLT}.  In particular,  
    \begin{align*}
	\Bigg(\frac{\eta(\tilde\Lambda_n^b(nz))-n^d\prod_{i=1}^db_i}{\sqrt{\mathrm{Var}\big(\eta(\tilde\Lambda_n^b)\big)}}\Bigg)_{z\in\mathbb{R}^d}{\underset{n\to\infty}\longrightarrow}\big(G^b(z)\big)_{z\in\mathbb{R}^d} 
    \end{align*}
    towards a centered stationary Gaussian field $G^b$ with covariance function $\mathrm{cov}^b(z)$ defined in Remark \ref{rem:rectangles_integrable} and Remark \ref{rem:rectangles_non-integrable} for the multidimensional case or $\mathrm{cov}^1$ in the one dimensional case.
\end{rmk}
\begin{proof}[Proof of Theorem \ref{thm:CLT}]
	
	Fix $N\in\mathbb{N}$ and $z_1,\dots,z_N\in\mathbb{R}^d$. In order to prove the weak convergence
	$$ \mathrm{Var}(\eta(\Lambda_n))^{-1/2}\big(\eta(\Lambda_n(nz_1))-n^d,\dots,\eta(\Lambda_n(nz_N))-n^d\big) \rightarrow (G(z_1),\dots,G(z_N))
	$$
	via the Cram\'er Wold device, it suffices to show that for all $t_1,\dots,t_N>0$ and $\varphi_n(x)=\sum_{i=1}^Nt_i\mathbb 1_{\Lambda_n(nz_i)}(x)$ we have
	$$
	\frac {\sum_{i=1}^Nt_i\big(\eta(\Lambda_n(nz_i))-n^d\big)} {\sqrt{\mathrm{Var}(\eta(\Lambda_n))}}
	=\frac {\eta(\varphi_n)-n^d \sum_i t_i}{\sqrt{\mathrm{Var}(\eta(\Lambda_n))}}
	\rightarrow \sum_{i=1}^Nt_iG(z_i).
	$$
	To this end, we shall calculate the cumulants $\mathrm{Cum}_k$ and show that cumulants of high order $k$ vanish. The claim then follows from Marcinkiewicz theorem, stating that only the Gaussian distribution has a finite number of nonzero cumulants.
	The first cumulant is the mean, which is equal zero. The second cumulant is the variance, giving
	\begin{align*}
		\mathrm{Cum}_2\Big(\frac {\eta(\varphi_n)-n^d \sum_i t_i}{\sqrt{\mathrm{Var}(\eta(\Lambda_n))}}\Big)=\frac{\mathrm{Var}(\eta(\varphi_n))}{\mathrm{Var}(\eta(\Lambda_n))}=\sum_{i,j=1}^Nt_it_j\frac{\mathrm{Cov}\big(\eta(\Lambda_n(nz_i)),\eta(\Lambda_n(nz_j))\big)}{\mathrm{Var}(\eta(\Lambda_n))}\to \sum_{i,j=1}^Nt_it_j\mathrm{cov}(z_i-z_j)
	\end{align*}
	by stationarity and our previous results.
	
	By our assumption, we have $\mathrm{Var}(\eta(\Lambda_n))\gtrsim n^{\epsilon}$ for some $\epsilon>0$: Indeed, for $d\ge 2$ the claim holds for $\epsilon=d-1$ by hyperuniformity and, for $d=1$ we assumed $a>0$, meaning $\mathrm{Var}(\eta(\Lambda_y))$ to be regularly varying with $a>0$ so that the claim holds with $\epsilon =a$. Since cumulants of order $k\ge 2$ are shift invariant and $k$-homogeneous, it follows
	\begin{align}\label{eq:cum_phi}
		\mathrm{Cum}_k\Big(\frac {\eta(\varphi_n)-n^d \sum_i t_i}{\sqrt{\mathrm{Var}(\eta(\Lambda_n))}}\Big)\lesssim n^{-k\epsilon /2} \mathrm{Cum}_k(\eta(\varphi_n)).
	\end{align}
	This is asymptotically negligible for $k$ sufficiently large, if we verify that the cumulant is of order $\mathcal O(n^d)$. Indeed, let us choose $c>0$ large enough such that $\Lambda_n(nz_i)\subseteq [-cn,cn]^d$ for all $i=1,\dots,N$, then it follows from boundedness of $\varphi_n$ that
	$$
	|\mathrm{Cum}_k(\eta(\varphi_n))|
	=\Big\lvert\int \prod_{j=1}^k\varphi_n(x_j)\gamma^{(k)} (\mathrm{d}x_1,\dots,\mathrm{d}x_k)\Big\rvert
	\lesssim \gamma^{(k)}([-cn,cn]^{dk}),
	$$
	where the implicit constant in $\lesssim$ only depends on  $k,N,t_i,z_i,\eta$ but not on $n$. The Brillinger-mixing condition and  \eqref{eq:truncated_cumulant} imply that $\gamma^{(k)}([-cn,cn]^{dk})=\mathcal O( n^d)$, which completes the proof. 
	
	Given a fractional Brownian motion $B_{a/2}$, the claim
	$$\mathbb{E}[(B_{a/2}(z+1)-B_{a/2}(z))(B_{a/2}(w+1)-B_{a/2}(w))]=\mathrm{cov}(z-w)$$
	follows from a direct calculation using the covariance $\mathbb{E}[B_{a/2}(z)B_{a/2}(w)]=\frac1 2 (|z|^a +|w|^a -|z-w|^a)$.
\end{proof} 

\begin{rmk}
	Lifting the convergence of finite-dimensional distributions in the statement of Theorem \ref{thm:CLT} to path convergence in the Skorokhod space (or continuous paths in the non-integrable case) seems only possible with a delicate improved analysis of dependence on $z$ in the error terms of our proof and additional assumptions on the explicit expression of $\beta$. We would need to control \eqref{eq:Cov_non_int2} near $z\approx 0$, corresponding to a uniform convergence of the defining limit of regular variation \eqref{eq:regvarr}, which may depend strongly on $\beta$.
\end{rmk}

\section*{Acknowledgements}
We thank Thomas Lebl\'e for bringing up the question of this work and Martin Huesmann for helpful discussions and valuable feedback. Our thanks also go to the anonymous referees for their careful reading and constructive suggestions. HS was supported by the Deutsche Forschungsgemeinschaft (DFG, German Research Foundation) under Germany's Excellence Strategy EXC 2044/2 –390685587, Mathematics M\"unster: \emph{Dynamics-Geometry-Structure}. JJ has been supported by the DFG priority program SPP 2265 \emph{Random Geometric Systems}.
\bibliography{Quel}
\bibliographystyle{alpha}

\end{document}